\documentclass[a4paper,11pt]{article}

\usepackage{amsfonts,amsmath,amssymb,amsthm}
\usepackage[applemac]{inputenc}
\usepackage[T1]{fontenc}
\usepackage[english]{babel}
\usepackage[all]{xy}
\usepackage{verbatim}
\usepackage{hyperref}
\usepackage{makeidx}
\newcommand{\Q}{\mathbb{Q}}
\newcommand{\C}{\mathbb{C}}
\newcommand{\Z}{\mathbb{Z}}

\newcommand{\G}{\mathbb{G}}
\newcommand{\M}{\mathcal{M}}

\renewcommand{\k}{{\mathnormal{k}}}
\newcommand{\kk}{\bar{\k}}
\DeclareMathOperator{\Gal}{Gal}

\DeclareMathOperator{\Hom}{Hom}

\DeclareMathOperator{\id}{id}
\DeclareMathOperator{\Spec}{Spec}

\DeclareMathOperator{\Spf}{Spf}
\DeclareMathOperator{\gr}{gr}
\DeclareMathOperator{\Lie}{Lie}

\DeclareMathOperator{\Coim}{Coim}
\DeclareMathOperator{\Ext}{Ext}
\DeclareMathOperator{\Ker}{Ker}
\DeclareMathOperator{\Coker}{Coker}

\renewcommand{\Im}{\mathrm{Im}}

\DeclareMathOperator*{\colim}{colim}
\newcommand{\x}{\times}
\newcommand{\ox}{\otimes}
\newcommand{\iso}{\cong}

\newcommand{\inv}{^{-1}}

\newcommand{\dfni}[1]{\emph{#1}\index{#1}}

\newtheorem{thr}{Theorem}[section]
\newtheorem{lmm}[thr]{Lemma}
\newtheorem{prp}[thr]{Proposition}

\theoremstyle{definition}\newtheorem{dfn}[thr]{Definition}

\theoremstyle{remark}\newtheorem{rmk}[thr]{Remark}
\theoremstyle{remark}

\title{Cohomological Dimension of \\ Laumon 1-motives up to Isogenies}
\author{Nicola Mazzari}
\date{\today}

\begin{document}
\maketitle
	\noindent
%	Nicola Mazzari\\
	Dipartimento di Matematica ``Federigo Enriques''\\
	Universit\`a degli Studi di Milano\\
	Via  Saldini, 50\\
	20133 - Milano\\
	E-Mail: {\tt mazzari@mat.unimi.it}\\[1cm]
	\begin{abstract}
		We prove that  the category of Laumon 1-motives up isogenies over a field of characteristic zero is of cohomological dimension $\le 1$. As a consequence this implies the same result for the category of formal Hodge structures of level $\le 1$ (over $\Q$).
	\end{abstract}
\tableofcontents
\section*{Introduction}
In  \cite{deligne:hodge3} Deligne defined  1-motive over a field $\k$ as $\Gal(\k^{\rm sep}|\k)$-equivariant morphism $[u:\boldsymbol{X}\to \boldsymbol{G}(\k^{\rm sep})]$ where $\boldsymbol{X}$ is a free $\Gal(\k^{\rm sep}|\k)$-module and $\boldsymbol{G}$ is a semi-abelian algebraic group over $\k$.  They form a category that we shall denote by $\M_{1,\k}$ or $\M_1$. 

Deligne's definition was motivated by Hodge theory. In fact the category of 1-motives over the complex numbers is equivalent, via the so called \emph{Hodge realization} functor, to the category  $\sf MHS_1$ of mixed Hodge structures of level $\le 1$. It is known the the category $\sf MHS_1$ is of cohomological dimension $1$ (see \cite{beilinson:abshodge}) and the same holds for $\M_{1,\C}$. \\
F. Orgogozo proved more generally that for any field $\k$, the category $\M_{1,\k}\ox \Q$  is of cohomological dimension $\le 1$ (see \cite[Prop. 3.2.4]{orgogozo:isomotif}).  

Over a field of characteristic  $0$  it is possible to define the category $\M_{1,\k}^{\rm a}$ of Laumon 1-motives generalizing that of Deligne 1-motives (See \cite{laumon}). In \cite{bv:fht} L.~Barbieri-Viale generalized the Hodge realization functor to Laumon 1-motives. He defined the category $\sf FHS_1$ of formal Hodge structures of level $\le 1$ containing $\sf \sf MHS_1$ and proved that $\sf FHS_1$ is equivalent to the category of Laumon 1-motives over $\C$ (compatibly with the Hodge realization).

In this paper we prove that the category of Laumon 1-motives up to isogenies is of cohomological dimension $1$.

\subsection*{Acknowledgments}
The author would like to thank L. Barbieri-Viale for 
pointing his attention to this subject and  for helpful discussions. The author also thanks A. Bertapelle for many useful comments and suggestions.

	\section{Laumon 1-motives}
	In this paper $\k$ is a field of characteristic $0$ and $\kk$ is its algebraic closure. As explained in the \S \ref{sec:fppf} we can assume that the categories of formal and algebraic groups are full sub-category of $\sf Ab_\k$, i.e. the category of abelian sheaves on the category of affine $\k$-schemes w.r.t. the fppf topology.
	\begin{dfn}
		A \dfni{Laumon 1-motive} over $\k$ (or an effective free 1-motive over $\k$, cf. \cite[1.4.1]{bv.bertapelle:sharpderham}) is the data of

		i) A (commutative) formal group $\boldsymbol{F}$ over $\k$, such that $\Lie \boldsymbol{F}$ is a finitely generated and $\boldsymbol{F}(\kk)=\lim_{[\k':\k]<\infty}\boldsymbol{F}(\k')$ is finitely generated and torsion-free $\Gal(\kk/\k)$-module. 

		ii) A connected commutative algebraic group scheme $\boldsymbol{G}$ over $\k$.

		iii) A morphism $u:\boldsymbol{F}\to \boldsymbol{G}$ in the category $\sf Ab_\k$.
	\end{dfn}
	Note that we can consider a Laumon 1-motive (over $\k$) $M=[u:\boldsymbol{F}\to \boldsymbol{G}]$ as a complex of sheaves in $\sf Ab_\k$ concentrated in degree $0,1$. \\
	It is known that any formal $\k$-group $\boldsymbol{F}$ splits canonically as product $\boldsymbol{F}^o \times \boldsymbol{F}_{\rm et} $ where $\boldsymbol{F}^o$ is the identity component of $\boldsymbol{F}$ and is a connected formal $\k$-group, and $\boldsymbol{F}_{\rm et} = \boldsymbol{F} /\boldsymbol{F}^o$ is \'etale. Moreover, $\boldsymbol{F}_{\rm et}$ admits a maximal sub-group scheme $\boldsymbol{F}_{\rm tor}$ , \'etale and finite, such that the quotient $\boldsymbol{F}_{\rm et} /\boldsymbol{F}_{\rm tor} = \boldsymbol{F}_{\rm fr}$ is constant of the type $\Z^r$ over $\kk$. One says that $\boldsymbol{F}$ is torsion-free if $\boldsymbol{F}_{\rm tor} = 0$. 

	By a theorem of Chevalley any connected algebraic group scheme $\boldsymbol{G}$ is extension of an abelian variety $\boldsymbol{A}$ by a linear $\k$-group scheme $\boldsymbol{L}$ that is product of its maximal sub-torus $\boldsymbol{T}$ with a vector $\k$-group scheme $\boldsymbol{V}$. (See \cite{MR0344261} for more details  on algebraic and formal groups)
	\begin{dfn}
		A \dfni{morphism} of Laumon 1-motives is a commutative square in the category $\sf Ab_\k$. We denote by $\M^{\rm a}_1=\M^{\rm a}_{1,\k}$ the category of Laumon $\k$-1-motives, i.e. the full sub-category of $C^b(\sf Ab_\k)$ whose objects are Laumon 1-motives.
	\end{dfn}
\begin{rmk}
	The category of Deligne 1-motives (over $\k$) is the full sub-category $\M_1$ of $\M_1^{\rm a}$ whose objects are $M=[u:\boldsymbol{F}\to \boldsymbol{G}]$ such that $\boldsymbol{F}^o=0$ and $\boldsymbol{G}$  is semi-abelian (cf. \cite[\S 10.1.2]{deligne:hodge3}).
\end{rmk}	
\begin{prp}
	The category $\M_1^{\rm a}$ of Laumon 1-motives (over $\k$) is an additive category with kernels and co-kernels.
\end{prp}
\begin{proof}
	See \cite[Prop. 5.1.3]{laumon}.
\end{proof}
\begin{rmk}
	i)  Let 
	\begin{equation*}
	\xymatrix{
	\boldsymbol{F} \ar[d]_{u}\ar[r]^{f}& \boldsymbol{F}'   \ar[d]^{u'}\\
\boldsymbol{G}	\ar[r]_{g} & \boldsymbol{G}' }
	\end{equation*}
	be a morphism from $M=[u:\boldsymbol{F}\to \boldsymbol{G}]$ to $M'=[u':\boldsymbol{F}'\to \boldsymbol{G}']$. Then from previous proof we get
	\begin{equation}\label{eq:ker1mot}
		\Ker(f,g)=[u^*\Ker(g)^o\to \Ker(g)^o]
	\end{equation}
	where $u:\Ker(f)\to \Ker(g)$, and
	\begin{equation}\label{eq:coker1mot}
		\Coker(f,g)=[\Coker(f)_{\rm fr}\to \Coker (g)]
	\end{equation}
	
ii) The category of Laumon 1-motives is not abelian. In fact consider a surjective morphism of connected algebraic groups $g:\boldsymbol{G}\to \boldsymbol{G}'$. Then $\Ker(g)$ is not necessarily connected. Hence in the category of connected algebraic groups the canonical map 
\[
	\Coim(g)=\boldsymbol{G}/\Ker(g)^o\rightarrow \Im(g)=\boldsymbol{G}'
\]
is not an isomorphism in general.

Note that the category of connected algebraic groups is fully embedded in $\M_1^{\rm a}$.
\end{rmk}
According to \cite{orgogozo:isomotif} we  define the  category ${\cal M}^{\rm a}_1\ox\Q$ of  Laumon 1-motives up to isogenies: the objects are the same of ${\cal M}^{\rm a}_1$; the Hom groups are $\Hom_{{\cal M}^{\rm a}_1}(M,M')\ox_\Z \Q $.
\begin{rmk}
	Note that a morphism $(f,g):M\to M'$ is an isogeny (i.e. an isomorphism in ${\cal M}^{\rm a}_1\ox \Q$ ) if and only if $f$ is injective with finite co-kernel and $g$ is surjective with finite kernel.
\end{rmk}
\begin{prp}
	The category of Laumon 1-motives up to isogenies is  abelian. 
\end{prp}
\begin{proof}
	By construction ${\cal M}^{\rm a}_1\ox \Q$ is an additive category. Let $(f,g):M\to M'$ be a morphism of Laumon 1-motives. 
	We know that  the group $\pi_0(\Ker(g))=\Ker(g)/\Ker(g)^o$ is a finite group scheme, hence there exists an integer $n$ such that the following diagram commutes in $\sf Ab_\k$
	\begin{equation*}
	\xymatrix{
	& \Ker (f)\ar[d]^{n\cdot u}\ar[dr]^0\\
	\Ker(g)^o\ar[r]&\Ker(g)\ar[r]&\pi_0(\Ker(g))  }
	\end{equation*}
	Hence $n\cdot u $ factors through $\Ker(g)^o$. Then it is  easy to check that  $\Ker((f,g))=[(u^*\Ker(g)^o)\to \Ker(g)^o]$ is isogenous to $[\Ker(f)\to \Ker(g)^0]$. 
	
 It follows that $\Coim(f,g)$ is isogenous to $[(\boldsymbol{F}/\Ker(f))_{\rm fr}\to \boldsymbol{G}/\Ker(g)]$. As $G/\Ker(g)^o\to G/\Ker(g)$ is an isogeny we get that the canonical map $\Coim(f,g)\to \Im(f,g)$ an isogeny too.

This is enough to prove that the category 	${\cal M}^{\rm a}_1\ox \Q$ is abelian.
\end{proof}
\begin{rmk}
	We can define the category ${}^t\M_1^{\rm a}$ of \emph{1-motives with torsion} (over $\k$) as the full sub-category of $D^b(\sf Ab_\k)$ with objects complexes $[u:\boldsymbol{F}\to \boldsymbol{G}]$ concentrated in degree $0,1$ such that

		i) A (commutative) formal group $\boldsymbol{F}$ over $\k$, such that $\Lie \boldsymbol{F}$ is a finitely generated and $\boldsymbol{F}(\kk)=\lim_{[\k':\k]<\infty}\boldsymbol{F}(\k')$ is finitely generated (non necessarily torsion-free!) $\Gal(\kk|\k)$-module. 

		ii) A connected commutative algebraic group scheme $\boldsymbol{G}$ over $\k$.

		iii) A morphism $u:\boldsymbol{F}\to \boldsymbol{G}$ in the category $\sf Ab_\k$.
		
		Also we denote by ${}^t\M_1\subset {}^t\M_1^{\rm a}$ the full sub-category whose object are of the form $[u:\boldsymbol{F}\to \boldsymbol{G}]$ with $\boldsymbol{F}^o=0$ and $\boldsymbol{G}$ semi-abelian.
		
		In \cite[C.7.3]{bv-kahn:D1mot1} is proven that the canonical functor $\M_1\to {}^t\M_1$ induces an equivalence of the same categories up to isogeny, i.e. $\M_1\ox \Q\iso {}^t\M_1\ox \Q$. The same result holds for Laumon 1-motives, in fact all the arguments given in \cite{bv-kahn:D1mot1} work also in this setting. Hence there is an equivalence of categories 
		\[
			\M_1^{\rm a}\ox \Q \stackrel{\sim}{\longrightarrow} {}^t\M_1^{\rm a}\ox \Q\ .
		\] 
\end{rmk}
A Deligne 1-motive is endowed with an increasing filtration (of sub-1-motives) called the weight filtration (\cite[\S 10.1.4]{deligne:hodge3}) defined as follows
\[
	W_i=W_iM:=\begin{cases}
		[u:\boldsymbol{X}\to \boldsymbol{G}]& i\ge 0\\
		[0\to \boldsymbol{G}]& i=-1\\
		[0\to \boldsymbol{T}]& i=-2\\
		[0\to 0]&i\le -3
	\end{cases}
\]
hence we get 
\[
	\gr_i^WM=\begin{cases}
		[\boldsymbol{X}\to 0]&i= 0\\
		[0\to \boldsymbol{A}]& i=-1\\
		[0\to \boldsymbol{T}]& i=-2\\
		[0\to 0]&\text{otherwise}
	\end{cases}
\]

	According to \cite[C.11.1]{bv-kahn:D1mot1} we extend the weight filtration to   Laumon 1-motives. Let $M=[u:\boldsymbol{F}\to \boldsymbol{G}]$   be an Laumon 1-motive, then
	\[
		W_{-3}=0\ \subset W_{-2}=[0\to \boldsymbol{L}]\ \subset W_{-1}=[0\to \boldsymbol{G}]\ \subset W_{0}=M\ .
	\]
	\subsection{fppf sheaves}\label{sec:fppf}
	Let $\sf Sch_\k$ be the category of schemes over $\k$ and $\sf Aff_\k$ be the full sub-category of affine schemes. According to \cite[Exp. IV \S 6.3]{sga3.1} the fppf topology on $\sf Sch_\k$ is the one generated by: the families of jointly surjective open immersions in $\sf Sch_\k$; the finite families of jointly surjective, flat, of finite presentation and quasi-finite morphisms in $\sf Aff_\k$.

	Recall that $\sf Ab_\k$ is the category of abelian sheaves on $\sf Aff_\k$ w.r.t. the fppf topology.

	\begin{prp}\label{prp:fppfemb}
		i) The category of commutative group schemes over $\k$ is a full sub-category of $\sf Ab_\k$ via the functor of points $\boldsymbol{G}\mapsto h_{\boldsymbol{G}}:=\Hom_{\sf sch_\k}(-,\boldsymbol{G})$.

		ii) Let $\rm char(\k)=0$. The category of formal group schemes  is a full sub-category  of $\sf Ab_\k$ via the functor of points $\boldsymbol{F}=\Spf(A)\mapsto h_{\boldsymbol{F}}:=\Hom_{\sf alg_\k}^{\rm cont}(A,-)$ (where $\Hom_{\sf alg_\k}^{\rm cont}(-,-)$ denote the set of continuous homomorphisms of $\k$-algebras).
	\end{prp}
	\begin{proof}
		By a result of Grothendieck (\cite[Part I, \S 2.3.6]{fgaexplained}) every scheme (over $\k$) is a sheaf (on sets) w.r.t. the fppf topology on $\sf Sch_\k$. Hence it is also a fppf-sheaf on the sub-category $\sf Aff_\k\subset \sf Sch_\k$. From this fact (i) and (ii) follow for \'etale formal groups. 

	By the decomposition theorem for formal groups over a perfect field \cite{MR0344261} it remains to prove that any connected (or local) formal group is a sheaf. It is sufficient to note that 
	\[
		\widehat{\G}_a\iso\colim_n \Spec (\k[t]/(t^{n+1}))
	\]
	is a direct limit of affine schemes, hence a direct limit of sheaves of sets w.r.t. the fppf topology.
	\end{proof}
\section{Extensions}
\subsection{The group of $n$-extensions}

Let $\mathsf{A}$ be any abelian category (we don't suppose it has enough injective objects), then we can define its derived category $D(\mathsf{A})$ and the group of n-fold extension classes
$$
\Ext^{n}_{\mathsf{A}}(A,B):=\Hom_{D(\mathsf{A})}(A,B[n])\qquad A,B\in \mathsf{A}\ .
$$
As usual we identify this group with the group of classes of \emph{Yoneda extensions}, i.e. the set of exact sequences
$$
0\to B\to E_{1}\to\cdots \to E_n\to A\to 0
$$
modulo congruences (See \cite{iversen} or \cite{gelfand-manin}).
\subsubsection{A lemma on 2-fold extensions}

Consider a 2-fold extension $\gamma\in \Ext^2_{\sf A}(M,M')$.  It is represented by an exact sequence
\begin{equation}\label{eq:2ext}
		0\to M'\to E_1\to E_2 \to M\to 0\ .
\end{equation}
This can be written as the product of two 1-fold extensions as follows. Let $E:=\Ker (E_2\to M)=\Coker(M'\to E_1)$, then let $\gamma_1\in \Ext^1_{\sf A}(E,M')$, $\gamma_2\in \Ext^1_{\sf A}(M,E)$ be the classes represented by
\begin{equation}\label{eq:spliced}
		0\to M'\to E_1\to E\to 0\ ,\quad 0\to E\to E_2\to M\to 0
\end{equation}
respectively. Then $\gamma=\gamma_1\cdot \gamma_2$.

As a particular case, consider $W_{-2}\subset W_{-1}\subset W_{0}$ a sequence of objects of $\sf A$. We have the following exact sequences
\[
\gamma:\quad 0\to W_{-2} \to W_{-1} \to W_{0}/W_{-2}\to W_{0}/W_{-1}\to 0	
\]
\[
	\gamma_1:\qquad  0\to W_{-2} \to W_{-1} \to W_{-1}/W_{-2}\to 0 
\]
\[
	\gamma_2:\qquad  0\to W_{-1}/W_{-2} \to W_{0}/W_{-2}\to W_{0}/W_{-1}\to 0
\]
and  $\gamma=\gamma_1\cdot \gamma_2\in \Ext^2_{\sf A}(W_{0}/W_{-1},W_{-2})$. In this particular case we get
\begin{lmm}\label{lmm:ext}
	$\gamma=0$ in $\Ext^2_{\sf A}(W_{0}/W_{-1},W_{-2})$.
\end{lmm}
\begin{proof}
	See \cite[Lemma 3.2.5]{orgogozo:isomotif}, or \cite[p. 184]{gelfand-manin}.
\end{proof}
\subsection{Ext of 1-motives up to isogenies}\label{sec:ext1isomot}
From now on we call 1-motive a Laumon 1-motive (over $\k$) and $\Ext_\Q^i(M,M')$ is the group of classes of i-fold extensions in ${\cal M}^{\rm a}_1\ox \Q$.% i.e. the abelian category of Laumon 1-motives up to isogeny. 

We are going to prove that ${\cal M}^{\rm a}_1\ox \Q$ is of cohomological dimension $\le 1$. We start with the following result.
\begin{lmm}
	Let $\gr^W_i\M_1^{\rm a}$ be the full sub-category of $\M_1^{\rm a}$ with objects 1-motives pure of weight $i$, $i=0,-1,-2$. Then $\gr^W_i\M_1^{\rm a}\ox \Q$  is an abelian thick sub-category of  $\M_1^{\rm a}\ox \Q$  and it is of cohomological dimension  $0$.
\end{lmm}
\begin{proof}
	First we consider the case $M=\boldsymbol{F}[1],M'=\boldsymbol{F}'[1]$ pure of weight $0$ (i.e. formal groups). Let $0\to \boldsymbol{F}'[1]\to E\to \boldsymbol{F}[1]\to 0$ an exact sequence of 1-motives modulo isogenies. Then $E$ is also of weight $0$ (this follows directly from the definitions). Hence $\Ext^1_\Q(\boldsymbol{F}[1],\boldsymbol{F}'[1])$ is isomorphic to the group of classes of extensions in the category of formal groups over $\k$ modulo isogenies. We know that of $\k$-vector spaces ${\sf Mod}_\k$ is semi-simple, and so is the $\Q$-linearized category of free $\Gal(\kk/\k)$-modules, ${\sf Mod}^{\rm free}_{\Gal(\kk/\k)}\ox \Q$, by the lemma of  Maschke (See \cite[p. 47]{serre:linrep}, for the representations of finite groups; the case of pro-finite is a direct consequence). Hence the category of formal groups up to isogeny, equivalent to ${\sf Mod}_\k\x {\sf Mod}^{\rm free}_{\Gal(\kk/\k)}\ox \Q$, is of cohomological dimension $0$.
	
	The second case is that of abelian varieties (weight $-1$). Again using the definitions  we get that $\Ext_\Q^1(\boldsymbol{A}',\boldsymbol{A})$ correspond to the group of extensions in the category of abelian varieties modulo isogenies. This group is zero (See \cite[p. 173]{mumford:abvar}).
	
	The third case is that of linear groups (weight $-2$). This can be reduced to the first case by Cartier duality (See \cite[\S 5]{laumon}) or proved explicitly.
\end{proof}
\begin{lmm}
	Let  $M,M'$ be pure Laumon 1-motives  with weights $w<w'$. Then  $\Ext_\Q^2(M,M')=0$.
\end{lmm}
\begin{proof}
	Fix  a 2-fold extension $\gamma\in \Ext_\Q^2(M,M')$ represented by
		\[
			0\to M'\to E_1\to E_2\to M\to 0
		\]
	and take $\gamma_2\in \Ext_\Q^1(M,E)$, $\gamma_1\in \Ext_\Q^1(E,M')$ (as in \eqref{eq:2ext}, \eqref{eq:spliced}) such that $\gamma=\gamma_1\cdot \gamma_2$. 
	
	We have an exact sequence
\[
	 \Ext_\Q^1(M,W_{-1}E)\to\Ext_\Q^1(M,E)\to \Ext_\Q^1(M,\gr_0 E)
\]
 By assumption  $-2\le w<w'\le 0$, then $M$  is pure of weight $-1$ or $-2$.
 In this case we get easily $\Ext_\Q^1(M,\gr_0 E)=0$ and we can lift $\gamma_2$ to $\gamma_2'\in \Ext_\Q^1(M,W_{-1}E)$. Then let $\gamma_1'$ be the image of $\gamma_1$ via $\Ext_\Q^1(E,M')\to \Ext_\Q^1(W_{-1}E,M')$. Now using
\[
	\Ext_\Q^1(\gr_{-1}E,M')\to\Ext_\Q^1(W_{-1}E,M')\to \Ext_\Q^1(W_{-2} E,M')
\]
we can reduce to consider $E$ pure of weight $-1$, in fact $\Ext_\Q^1(W_{-2} E,M')=0$ because $w'>-2$. From this follows that $\gamma_1=\gamma_2=0$.
\end{proof}
\begin{lmm}
		Let  $M,M'$ be pure Laumon 1-motives  with weights $w>w'$. Then  $\Ext_\Q^2(M,M')=0$.
\end{lmm}
\begin{proof}
	As in the previous proof fix  a 2-fold extension $\gamma\in \Ext_\Q^2(M,M')$ represented by
		\[
			0\to M'\to E_1\to E_2\to M\to 0
		\]
	and take $\gamma_2\in \Ext_\Q^1(M,E)$, $\gamma_1\in \Ext_\Q^1(E,M')$ (as in \ref{eq:2ext}, \ref{eq:spliced}) such that $\gamma=\gamma_1\cdot \gamma_2$.
	
	We have to consider three cases: $(a)$ $ M=\boldsymbol{F}[1]$, $M'=\boldsymbol{A}[0]$; $(b)$ $M=\boldsymbol{F}[1]$, $M'=\boldsymbol{L}[0]$; $(c)$ $M=\boldsymbol{A}[0]$, $M'=\boldsymbol{L}[0]$. Where $\boldsymbol{F}$ is a formal group, $\boldsymbol{A}$ an abelian variety, $\boldsymbol{L}$ a linear group.

		Case $(a)$: now  $\gamma_1\in \Ext_\Q^1(E,\boldsymbol{A})$ $\gamma_2\in\Ext_\Q^1(\boldsymbol{F}[1],E)$. Then $E=[\boldsymbol{F}'\to \boldsymbol{A}']$ is such that $W_{-2}E=0$. Consider the exact sequence
		\[
			0\to \gr_{-1} E\to E\to \gr_{0} E\to 0
		\]
		applying $\Hom_\Q(\boldsymbol{F}[1],-)$ to it we get
		\[
		\Ext_\Q^1(\boldsymbol{F}[1],\gr_{-1}E)\to	\Ext_\Q^1(\boldsymbol{F}[1],E)\to \Ext_\Q^1(\boldsymbol{F}[1],\gr_{0} E)
		\]
		We proved that $\Ext_\Q^1(\boldsymbol{F}[1],\gr_{0} E)=0$ so we can lift $\gamma_2$ to a class $\gamma_2'\in \Ext_\Q^1(\boldsymbol{F}[1],\gr_{-1}E)$ (This lifting is not canonical).
		Similarly using  $\Hom_\Q(-,\boldsymbol{A})$ to it we get an exact sequence
		\[
			\Ext_\Q^1(\gr_{0}E,\boldsymbol{A})\to	\Ext_\Q^1(E,\boldsymbol{A})\to \Ext_\Q^1(\gr_{-1} E,\boldsymbol{A})
		\]
		and we can map $\gamma_1\mapsto \gamma_1'\in \Ext_\Q^1(\gr_{-1} E, \boldsymbol{A})$. By standard arguments it holds $\gamma_1'\cdot\gamma_2'=\gamma_1\cdot\gamma_2=\gamma$. Recalling that $\Ext_\Q^1(\gr_{-1} E,\boldsymbol{A})=0$ we get the result.

		Case $(c)$: Is similar to case $(a)$.

		Case $(b)$: now $\gamma\in \Ext_\Q^2(\boldsymbol{F}[1],\boldsymbol{L})$. We want to reduce to the hypothesis of  lemma \ref{lmm:ext}. Thus we have to show: we can take $E$ pure of weight $1$ (i.e. an abelian variety);  there exists a 1-motive $N$ such that $\gamma_1\in \Ext_\Q^1(E,\boldsymbol{L})$ is represented by 	$0\to W_{-2} N\to W_{-1} N\to \gr_{-1} N\to 0$;  $\gamma_2\in \Ext_\Q^1(\boldsymbol{F}[1],E)$ is represented by 	$0\to \gr_{-1} N\to W_{0} N/W_{-2} \to \gr_{0} N\to 0$.\\
		We know that $\Ext_\Q^1(\boldsymbol{F}[1],\gr_{0} E)=0$, so like in case $(a)$ we can lift $\gamma_2$ to a class $\gamma_2'\in \Ext_\Q^1(\boldsymbol{F}[1],W_{-1} E)$. Let $\gamma_1'$ the image of $\gamma_1$ via $\Ext_\Q^1(E,\boldsymbol{L})\to \Ext_\Q^1(W_{-1}E,\boldsymbol{L})$. Hence $\gamma_1'\cdot \gamma_2'=\gamma_1\cdot \gamma_2$. \\
		Now we can suppose $E$ of weight $\le - 1$. Using the same argument we can lift $\gamma_1'$ to $\gamma_1''\in \Ext_\Q^1(\gr_{-1}E,\boldsymbol{L})$ (because $\Ext_\Q^1(\gr_{-2}E,\boldsymbol{L})=0$) and send $\gamma_2'\mapsto \gamma_2''\in \Ext_\Q^1(\boldsymbol{F}[1],\gr_{-1}E)$.
		We proved  that there exists an abelian variety $\boldsymbol{A}$, $\gamma_1\in \Ext_\Q^1(\boldsymbol{A},\boldsymbol{A})$, $\gamma_2\in \Ext_\Q^1(\boldsymbol{F}[1],\boldsymbol{A})$, such that $\gamma_1\cdot \gamma_2=\gamma$. We claim that $\gamma_1,\gamma_2$ can be represented by extensions in the category Laumon-1-motives. In fact let 
		\[
			\gamma_1:\quad 0 \to \boldsymbol{L} \stackrel{f\ox n\inv}{\longrightarrow}\boldsymbol{G} \stackrel{g\ox m\inv}{\longrightarrow} \boldsymbol{A}\to 0
		\]
		be an extension in the category of 1-motives modulo isogenies: $f,g$ are morphism of algebraic groups, $n,m\in \Z$. Then consider the push-forward by $n\inv$ and the pull-back by $m\inv$, we get the following commutative diagram with exact rows in ${\cal M}^{\rm a,fr}_1\ox \Q$
		\begin{equation*}
		\xymatrix{
		0\ar[r]&\boldsymbol{L}\ar[d]^{n\inv} \ar[r]^{f/n}&\boldsymbol{G}\ar[d]^\id\ar[r]^{g/m}&\boldsymbol{A}\ar[d]^\id\ar[r]&0\\ 
		0\ar[r]&\boldsymbol{L} \ar[r]^{f}&\boldsymbol{G}\ar[r]^{g/m}&\boldsymbol{A}\ar[r]&0\\
		0\ar[r]&\boldsymbol{L}\ar[u]_{\id} \ar[r]^{f}&\boldsymbol{G}\ar[u]_\id\ar[r]^{g}&\boldsymbol{A}\ar[u]_{m\inv}\ar[r]&0 }
		\end{equation*}

	The exactness of the last row is equivalent to the following: $\Ker f$ is finite; let $(\Ker g)^0$ be the connected component of $\Ker g$, then $\Im f\to (\Ker g)^0$ is surjective with finite kernel $K$; $g$ is surjective. So after replacing $\boldsymbol{L},\boldsymbol{A}$ with isogenous groups we have an exact sequence in ${\cal M}^{\rm a,fr}_1$
	\[
		0\to \boldsymbol{L}\to \boldsymbol{G}\to \boldsymbol{A}\to 0
	\]
	Explicitly
	\begin{equation*}
	\xymatrix{
	0\ar[r]&\boldsymbol{L}\ar[d] \ar[r]^{f}&\boldsymbol{G}\ar[d]^\id\ar[r]^{g}&\boldsymbol{A}\ar[d]^\id\ar[r]&0\\ 
	0\ar[r]&\boldsymbol{L}/\Ker f \ar[d]\ar[r]&\boldsymbol{G}\ar[d]\ar[r]^{g}&\boldsymbol{A}\ar[d]\ar[r]&0\\
	0\ar[r]&\Im f/K \ar[r]&\boldsymbol{G}\ar[r]^{g}&\boldsymbol{A}\ar[r]&0\\
	0\ar[r]&\Im f/K\ar[u]_{\id} \ar[r]&\boldsymbol{G}'\ar[u]\ar[r]&\boldsymbol{G}/(\Ker g)^0\ar[u]\ar[r]&0 }
	\end{equation*}

	With similar arguments we can prove that $\gamma_2$ is represented by an extension in the category ${\cal M}^{\rm a,fr}_1$
	\[
		0\to \boldsymbol{A}\to N\to \boldsymbol{F}[1]\to 0
	\]
	with $N=[u:\boldsymbol{F}\to \boldsymbol{A}]$.

	To apply  lemma \ref{lmm:ext} we need to prove that there is lifting $u':\boldsymbol{F}\to \boldsymbol{G}$. First suppose $\boldsymbol{F}=\boldsymbol{F}_{\rm et}$: consider the long exact sequence
	\[
		\Hom_{\sf Ab_\k}(\boldsymbol{F},\boldsymbol{G})\to \Hom_{\sf Ab_\k}(\boldsymbol{F},\boldsymbol{A})	\stackrel{\partial}{\rightarrow} \Ext^1_{\sf Ab_\k}(\boldsymbol{F},\boldsymbol{L})
	\]
	We know (\cite{milne:etcoho}) that $\Ext^1_{\sf Ab_\k}(\boldsymbol{F},\boldsymbol{L})$ is a torsion group. So modulo replacing $\boldsymbol{F}$ with an isogenous lattice we get $\partial u=0$ and the lift exists.\\
	In case $\boldsymbol{F}=\boldsymbol{F}^o$ is a connected formal group we have a commutative diagram in $\sf Ab_\k$
	\begin{equation*}
	\xymatrix{
	 & & & \ar[dl]\boldsymbol{F}\ar[dd]^u\\
	\widehat{\boldsymbol{G}}\ar[dr]\ar[rr]^{\widehat{\pi}}& &\widehat{\boldsymbol{A}} \ar[dr]\\
	& \boldsymbol{G}\ar[rr]^\pi& & \boldsymbol{A}}
	\end{equation*} 
	where $\widehat{?}$ is the formal completion at the origin of $?=\boldsymbol{G},\boldsymbol{A}$. %(See \ref{exm:formalLiegrp}).
	The formal completion is an exact functor so $\widehat{\pi}$ is an epimorphism. The category of formal groups is of cohomological dimension 0, then we can choose a section of $\widehat{\pi}$ and lift $u$.
\end{proof}
\begin{thr}
	The category ${\cal M}^{\rm a}_1\ox \Q$ is of cohomological dimension 1. 
\end{thr}
\begin{proof}
	First note that we  can restrict to consider  pure motives $M,M'$ (a 1-motive is pure if it is isomorphic to one of its graded pieces w.r.t. the weight filtration). In fact given $M,M'$ 1-motives, not necessarily pure, we have the canonical exact sequences given by the weight filtration
		\[
			0\to W_{-1} M'\to M'\to \gr_{0}^W M'\to 0\]
			\[ 0\to W_{-2} M'\to W_{-1}M'\to \gr_{-1}^W M'\to 0
		\]
	Hence	applying $\Hom_\Q(M,-)$ we get two long exact sequences
		\[
			\cdots \Ext_\Q^2(M,W_{-1}M')\to \Ext_\Q^2(M,M')\to\Ext_\Q^2(M,\gr_{0}^W M')\cdots 
			\]
			\[ \cdots \Ext_\Q^2(M, W_{-2} M')\to \Ext_\Q^2(M,W_{-1}M')\to \Ext_\Q^2(M,\gr_{-1}^W M')\cdots
		\]
		from this follows that we can reduce to prove $\Ext_\Q^2(M,M')=0$ for $M'$ pure. In the same way we reduce to consider $M$ pure.
	
	Then using the previous three lemmas we conclude.
\end{proof}
\addcontentsline{toc}{section}{References}
%\bibliographystyle{plain}
%\bibliography{tutto}
%

\end{document}